\documentclass[11pt]{article}

\usepackage[letterpaper, hmargin=1.3in, top=1in, bottom=1.3in, footskip=1in]{geometry}
\usepackage{float}

\usepackage{titlesec}
\titleformat{\subsection}[runin]{\normalfont\bfseries}{\thesubsection.}{.5em}{}[.]\titlespacing{\subsection}{0pt}{2ex plus .1ex minus .2ex}{.8em}
\titleformat{\subsubsection}[runin]{\normalfont\itshape}{\thesubsubsection.}{.3em}{}[.]\titlespacing{\subsubsection}{0pt}{1ex plus .1ex minus .2ex}{.5em}

\usepackage[labelfont=sc,font=small,labelsep=period]{caption}
\usepackage{amsmath} 
\usepackage{amssymb}
\usepackage{amsfonts}
\usepackage{latexsym}
\usepackage{amsthm}
\usepackage{amsxtra}
\usepackage{amscd}
\usepackage{bbm}
\usepackage{mathrsfs}
\usepackage{bm}
\usepackage{xcolor}
\usepackage{subcaption}
\usepackage[toc]{appendix}
\usepackage[utf8]{inputenc}

\usepackage{graphicx, color}

\definecolor{darkred}{rgb}{0.9,0,0.3}
\definecolor{darkblue}{rgb}{0,0.3,0.9}

\usepackage{ifthen}
\def\comment#1{\ifthenelse{\isodd{\value{page}}}{\marginpar{\raggedright\scriptsize{\textcolor{darkred}{#1}}}}{\marginpar{\raggedleft\scriptsize{\textcolor{darkred}{#1}}}}}  

\usepackage[nottoc,notlof,notlot]{tocbibind}
\usepackage{cite} 

\numberwithin{equation}{section}
\numberwithin{figure}{section}

\theoremstyle{plain} 
\newtheorem{theorem}{Theorem}[section]
\newtheorem*{theorem*}{Theorem}
\newtheorem{lemma}[theorem]{Lemma}
\newtheorem*{lemma*}{Lemma}
\newtheorem{corollary}[theorem]{Corollary}
\newtheorem*{corollary*}{Corollary}
\newtheorem{proposition}[theorem]{Proposition}
\newtheorem*{proposition*}{Proposition}
\newtheorem{definition}[theorem]{Definition}
\newtheorem*{definition*}{Definition}

\newtheorem*{conjecture*}{Conjecture}

\theoremstyle{definition} 

\newtheorem*{example*}{Example}
\newtheorem{remark}[theorem]{Remark}
\newtheorem*{remark*}{Remark}

\renewcommand{\leq}{\leqslant}
\renewcommand{\geq}{\geqslant}
\renewcommand{\epsilon}{\varepsilon}

\title{Quasi-Sure Stochastic Analysis through Aggregation and SLE$_\kappa$ Theory}
\author{Vlad Margarint}

\begin{document}

\maketitle

\begin{abstract}
We study SLE$_{\kappa}$ theory with elements of Quasi-Sure Stochastic Analysis through Aggregation. Specifically, we show how the latter can be used to construct the SLE$_{\kappa}$ traces quasi-surely (i.e. simultaneously for a family of probability measures with certain properties) for $\kappa \in \mathcal{K}\cap \mathbb{R}_+ \setminus ([0, \epsilon) \cup \{8\})$, for any $\epsilon>0$ with $\mathcal{K} \subset \mathbb{R}_{+}$ a nontrivial compact interval, i.e. for all $\kappa$ that are not in a neighborhood of zero and are different from $8$. As a by-product of the analysis, we show in this language a version of the continuity in $\kappa$ of the SLE$_{\kappa}$ traces for all $\kappa$ in compact intervals as above.
\end{abstract}

\section{Introduction}

The Loewner equation (also known as the Loewner evolution) was introduced by Charles Loewner in $1923$ in \cite{lowner1923untersuchungen} and it played an important role in the proof of the Bieberbach Conjecture \cite{bieberbach1916uber} by Louis de Branges in $1985$ in \cite{de1985proof}. In 2000, Oded Schramm introduced in \cite{schramm2000scaling} a stochastic version of the Loewner equation, the Stochastic Loewner Evolution ($SLE_\kappa$). The $SLE_{\kappa}$ describes the evolution of a curve in terms of a driving function that is chosen to be $\sqrt{\kappa}B_t$, with $\kappa \geq0$ a real parameter and $B_t, \hspace{1mm} t \in [0,\infty),$ a real-valued standard Brownian motion. This is a one-parameter family of random planar fractal curves that are the only possible conformally invariant scaling limits of interfaces of a number of discrete models that appear in planar Statistical Physics. In several cases, it was proved that indeed the interfaces converge to the $SLE_{\kappa}$ curves. We refer to \cite{lawler2008conformally} for a detailed study of the object and many of its properties.

%

 The problem of continuity of the traces generated by Lowener chains was studied in the context of chains driven by bounded variation drivers in \cite{shekhar2017remarks}, where the continuity of the traces generated by the Loewner chains was established. Also, the question appeared in \cite{lind2010collisions}, where the Loewner chains were driven by H\" older-$1/2$ functions with norm bounded by $\sigma$ with $\sigma < 4\,.$ In this context, the continuity of the corresponding traces was established with respect to the uniform topologies on the space of drivers and with respect to the same topology on the space of simple curves in $\mathbb{H}\,.$ Another paper that addressed a similar problem is \cite{sheffield2012strong}, in which the condition $||U||_{1/2}<4$ on the driver $U_t$ of the Loewner differenial equation is avoided at the cost of assuming some conditions on the limiting trace. Some stronger continuity results are obtained in  \cite{friz2017existence} under the assumption that the driver $U_t$ of the Loewner differential equation has finite energy, in the sense that $\dot{U}$ is square integrable. 
Also, the continuity in $\kappa$ of $SLE_{\kappa}$ was studied in terms of the topology of weak convergence for the associated probability  measures on the space of curves in \cite{kemppainen2012random} for curves in the upper half-plane and in \cite{karrila2018limits} for more general domains.

The question appears naturally when considering the solution of the corresponding welding problem in \cite{astala2011random}. In this paper it is proved that the trace obtained when solving the corresponding welding problem is continuous in a parameter that appears naturally in the setting. In the context of $SLE_{\kappa}$ traces the problem was studied in \cite{viklund2014continuity}, where the continuity in $\kappa$ of the $SLE_{\kappa}$ traces was proved for any $\kappa <2.1$. A stronger result is proved in \cite{friz2019regularity}, where the a.s. continuity in $\kappa$ of the SLE traces is proved for $\kappa <8/3.$

Our method relies on the Quasi-Sure Stochastic Analysis through Aggregation as constructed in \cite{soner2011quasi}. The construction in \cite{soner2011quasi} is suitable when one works with mutually singular probability measures. In the case when the measures are absolutely continuous, the situation becomes simpler since one can work under the nullsets of the dominating measure directly.
In \cite{soner2011quasi}, the authors work with a family of local martingale measures $\mathbb{P}_a$ indexed by a parameter $a$ such that under $\mathbb{P}_0$ the canonical process (coordinate process) is a Brownian Motion (BM). When considering the family of measures $\mathbb{P}_{a}$, the canonical process becomes under each $\mathbb{P}_a$ a local martingale with quadratic variation $a$. In \cite{soner2011quasi} it is further shown that if the family of local martingale measures satisfies certain assumptions, then one can define a notion of $\textit{aggregator}$. Further, this notion of aggregator is used to construct the Universal Brownian Motion $W_t=\int_0^t a_s^{-1/2}dB_s$, i.e. an aggregator for Brownian motion, that is a $\mathbb{P}_a$- BM under any $\mathbb{P}_a$ in the family of probability measures. These constructions are natural when one is interested in studying problems related to uncertain volatility in Financial Markets. Morever, the aggregation result can be transported from the Brownian driver to diffusion equations with strong solutions, i.e. one can construct an aggregator of solutions of such SDEs.
 In order to make the link with $SLE_{\kappa}$ theory, we use the Universal Brownian motion (see \cite{soner2011quasi}) as a driver for the Loewner differential equation. In our case, the role of the parameter $a$ will be played by the natural parameter $\kappa$ in the $SLE_{\kappa}$ theory, since this is the volatility in this setting.
Using this, one can construct SLE traces simultaneously quasi-surely, i.e. simultaneously for a family of measures $\mathbb{P}_{\kappa}$, for all $\kappa \in \mathcal{K}\cap \mathbb{R}_+ \setminus ([0, \epsilon)\cup \{8\})$, for any $\epsilon>0$ with $\mathcal{K} \subset \mathbb{R}_{+}$ a nontrivial compact interval, using the aggregated solution to a stochastic differential equation that appears in the analysis, and expressing the derivative of the conformal maps in terms of this aggregated solution. Using the Quasi-sure Stochastic Analysis through Aggregation method, one can view these models in a unified framework. 
Furthermore, in this setting one can show a version of the continuity in $\kappa$ of the $SLE_{\kappa}$ traces, that we call  quasi-sure continuity in $\kappa$, using an estimate between conformal maps solving the Loewner Differential Equation whose drivers are close to each other obtained in \cite{viklund2014continuity}. 

The paper is divided in several sections. In the first part of the paper, we construct quasi-surely the $SLE_{\kappa}$ traces and in the second part we prove the (quasi-sure) q.s. continuity in $\kappa$ for $\kappa \in \mathcal{K}\cap \mathbb{R}_+ \setminus ([0, \epsilon) \cup \{8\})$ of these objects. We remark that in the case  $\kappa \in [0, \epsilon]$ the a.s. continuity in $\kappa$ of the $SLE_{\kappa}$ is known from previous works (see \cite{viklund2014continuity} and \cite{friz2019regularity}).

\textbf{Acknowledgement:} 
I acknowledge the support of NYU-ECNU Institute of Mathematical Sciences at NYU Shanghai.
Also, I would like to thank especially Johannes Wiesel who introduced me to the Quasi-Sure Stochastic Analysis through Aggregation method and who read carefully previous versions of this manuscript and gave me very useful insights along the way as well as to Alex Mijatovic who I have been discussing around the topic and gave me useful information. Also, I would like to thank Yizheng Yuan, Dmitry Beliaev, Titus Lupu, Jianping Jiang, Ionel Popescu and Iulian Cimpean for useful discussions and suggestions. Special thanks also for Lukas Schoug who read the most recent version of the manuscript and gave me useful remarks.


\section{Preliminaries}

We start by introducing objects needed in our analysis. For the SLE theory the exposition is based on \cite{lawler2008conformally} and \cite{berestycki2014lectures} and for the Quasi-Sure Stochastic Analysis through Aggregation the exposition is based on \cite{soner2011quasi} which we refer to for more details.

\subsection{Introduction to $SLE_{\kappa}$ theory}

\indent An important object in the study of the Loewner differential equation is the $\mathbb{H}$-compact hull that is a bounded closed set in $\mathbb{H}$ such that its complement in $\mathbb{H}$ is simply connected.  To every compact $\mathbb{H}$-hull, that we typically denote by $K$ we associate a canonical conformal map $g_K : \mathbb{H}\setminus K \to \mathbb{H}$ that is called the \textit{mapping out function of $K$}.


Using the Riemann Mapping Theorem, we get uniqueness by imposing the \textit{hydrodynamic normalization} for $g_K$, i.e. we require that the mapping near infinity is of the form 

$$g_K(z)=z+\frac{a_K}{z}+O(|z|^{-2})\,, \hspace{3mm} |z| \to \infty\,. $$
\noindent
The coefficient $a_K$ that appears in the expansion at infinity of the mapping is called $\textit{half-plane capacity}\,.$ Throughout the paper we use the notation $\text{hcap}(K)$ for the half-plane capacity $a_K$. \par

We work with a family of growing compact hulls $K_t$ and denote $H_t\;:=\; \mathbb{H}\setminus K_t\,.$ Firstly, we define the radius of a hull to be $$rad(K)=\inf \{ r \geq 0 : K \subset  r \mathbb{D} \text{+ x for some x} \in \mathbb{R} \}\,.$$

\begin{definition}
Let $(K_t)_{t \geq 0}$ be a family of increasing $\mathbb{H}$-hulls, i.e. $K_s$ is contained in $K_t$ whenever $s<t\,.$ For $s<t\,,$ set $K_{s,t}=g_{K_s}(K_t\setminus K_s)\,.$ We say that $(K_t)_{t \geq 0}$ has the \textit{local growth property} if 
$$rad(K_{t, t+h}) \to 0 \hspace{2mm} \text{as} \hspace{2mm} h \to 0, \hspace{2mm} \text{uniformly on compacts in t.} $$
\end{definition}

The first connection between the family of growing compact $\mathbb{H}$-hulls and the real-valued path $(U_t)_{t \geq 0}$ is done in the following proposition.

\begin{proposition}[Proposition $7.1$ of \cite{berestycki2014lectures}]
Let $(K_t)_{t \geq 0}$ be an increasing family of compact $\mathbb{H}$-hulls having the local growth property. Then, $K_{t+}=K_t$ for all $t\,.$ Moreover, the mapping $t \mapsto \text{hcap}(K_t)$ is continuous and strictly increasing on $[0, \infty)\,.$ Moreover, for all $t \geq 0$, there is a unique $U_t \in \mathbb{R}$ such that $U_t \in \bar{K}_{t, t+h}\,,$ for all $h >0\,,$ and the process $(U_t)_{t \geq 0}$ is continuous.  
\end{proposition}
The map $t \mapsto hcap(K_t)/2$ is a non-decreasing homeomorphism on $[0, T)$ and by choosing $\tau $ to be the inverse of this homeomorphism, we obtain a new family of hulls $K'_t$  in a new parametrization such that $\text{hcap}(K'_t) =2t\,.$ This is the canonical parametrization that we use throughout the paper. We use the standard terminology for this, i.e parametrization by half-plane capacity.

In the following proposition, we introduce the Loewner differential equation starting from the family of growing compact hulls. The main idea is that the local growth property of the hulls gives a description in terms of a specific differential equation for the associated mapping out functions.

\begin{proposition}[Proposition $7.3$ of \cite{berestycki2014lectures}]
Let $(K_t)_{t \geq 0}$ be a family of increasing compact hulls in $\mathbb{H}$ satisfying the local growth property and that are parametrized by the halfplane capacity. Let $(U_t)_{t \geq 0}$ be its Loewner transform. Set $g_t=g_{K_t}$ and $T(z)= \inf\{t \geq 0 : z \in K_t\}\,.$ Then, for all $z \in \mathbb{H}\,,$ the function $(g_t(z): t \in [0, T(z))$ is differentiable with respect to $t$ and satisfies the Loewner differential equation
$$\dot{g}_t(z)=\frac{2}{g_t(z)-U_t}\,. $$
Moreover, if $T(z) < \infty$ then $g_t(z)-U_t \to 0$ as $t \to T(z)\,.$
\end{proposition}

The reverse situation is also true, i.e. from the driving function $U_t\,,$ we recover the family of growing compact $\mathbb{H}$-hulls.
 We have this result in the following Theorem. Note that in \cite{lawler2008conformally}, the result is stated for Loewner differential equation driven by measures on the real line. We state it only for the particular choice of measure on the real line $\mu_t=2\delta_{U_t}$.
\begin{theorem}[Theorem $4.6$ of \cite{lawler2008conformally}]
For all $z \in \mathbb{H} \setminus \{\zeta_0\}$, there is a unique time $T(z)\in (0, \infty]$ and a unique continuous map $(g_t(z): t \in [0, T(z))$ in $\mathbb{H}\setminus K_t=H_t$ such that, for all $ t \in [0, T(z))$ we have $g_t \neq U_t$ and
$$g_t(z)=z+\int_0^t \frac{2}{g_s(z)-U_s}ds\,, $$ and such that $ |g_t(z)-U_t| \to 0$ as $ t \to T(z)$ whenever $T(z) < \infty\,.$ Set $\zeta_0=0$ and define $$H_t=\{z \in \mathbb{H}: T(z)>t \} \,.$$
Then, for all $t \geq 0$ $H_t$ is open and $g_t :H_t \to \mathbb{H}$ is conformal onto $\mathbb H$. Moreover, the family of sets $K_t$= $\left( z \in \mathbb{H} : T(z) \leq t \right)$ is an increasing family of compact $\mathbb{H}$-hulls having the local growth property with $hcap(K_t)$=$2t\,,$ and $g_{K_t}=g_t,$ for all $t\,.$ Moreover, the driving function $U_t$ is the Loewner transform of $(K_t)_{t \geq 0}\,.$
\end{theorem}

In the $SLE_{\kappa}$ theory case, the driver is chosen to be $U_t=\sqrt{\kappa}B_t, $ where $B_t$ is a standard one-dimensional Brownian motion and $\kappa \in \mathbb{R}_+$. When studying this theory in the upper half-plane, one usually works with the following families of conformal maps.
\begin{enumerate} 
\item Partial differential equation version for the chordal $SLE_{\kappa}$ in the upper half-plane
\begin{equation}\label{4}
\partial_{t}f(t,z)=-\partial_{z}f(t,z)\frac{2}{z-\sqrt{\kappa}B_{t}}\,, \hspace{3mm} f(0,z)=z, z \in \mathbb{H}\,.
\end{equation}
\item Forward differential equation version for chordal $SLE_{\kappa}$ in the upper half-plane

\begin{equation}\label{5}
\partial_{t}g(t,z)=\frac{2}{g(t,z)-\sqrt{\kappa}B_{t}}\,, \hspace{10mm} g(0,z)=z, z \in \mathbb{H}\,.
\end{equation}

\item  Time reversal differential equation (backward) version for chordal $SLE_{\kappa}$ in the upper half-plane
\begin{equation}\label{6}
\partial_{t}h(t,z)=\frac{-2}{h(t,z)-\sqrt{\kappa}B_{t}}\,, \hspace{10mm} h(0,z)=z, z \in \mathbb{H}\,.
\end{equation}
\end{enumerate} 
There are connections between these three formulations for studying families of conformal maps. For example, at each instance of time $t \in [0,T]$ the map $z \to g_t(z)$ is the inverse of the map $z \to f_t(z)\,.$ 
The connection between the family of maps $h_t(z)$ and $g_t(z)$ is captured in the following lemma.

\begin{lemma}[Lemma 5.5 of \cite{kemppainen2017schramm}]
Let $h_t(z)$ be the solution to the backward Loewner differential equation with driving function $\sqrt{\kappa}B_t$ and let $f_t(z)$ be the solution of the partial differential equation version of the Loewner differential equation. Then, for any $t \in \mathbb{R}_+$, the function $z \to f_t(z+\sqrt{\kappa}B_t)-\sqrt{\kappa}B_t$ and $z \to h_t(z)$ have the same distribution.
\end{lemma}

In $SLE_{\kappa}$ theory, a fundamental object of study are the  $SLE_{\kappa}$ traces that we introduce in the following definition.
\begin{definition}
Let $g_t$ be the conformal maps solving the forward Loewner differential equation with $U_t=\sqrt{\kappa}B_t$ . The $SLE_{\kappa}$ trace is defined as
$$\gamma(t):=\lim_{y\to 0}\hat{g}_t^{-1}(iy),$$
where $\hat{g}_t^{-1}(iy)=g_t^{-1}(iy+\sqrt{\kappa}B_t).$
\end{definition}
For general Loewner chains, we have the following definition for hulls generated by a trace.
\begin{definition}
We say that a continuous path $(\gamma_t)_{t \geq 0}$ in $\bar{\mathbb{H}}$ generates a family of increasing compact $\mathbb{H}$-hulls $K_t$ if $H_t=\mathbb{H}\setminus K_t$ is the unbounded component of $\mathbb{H}\setminus \gamma[0,t]$ for all $t \geq 0\,.$ 
\end{definition}
When considering the $SLE_{\kappa}$ case, we have the following fundamental result.
\begin{theorem}[Theorem $4.1$ of \cite{schramm2005basic}]\label{Rohdeschramm}
Let $(K_t)_{t \geq 0}$ be a $SLE_{\kappa}$ for $\kappa \neq 8\,.$  Then, $$\hat{g}_{t}^{-1}(z)= g_{t}^{-1}(z+\sqrt{\kappa}B_t) :\mathbb{H} \mapsto H_t$$ extends continuously to $\bar{\mathbb{H}}$ for all $t \geq 0,$ almost surely. Moreover, $\gamma_t$ is continuous and generates $(K_t)_{t \geq 0}$ almost surely.
\end{theorem}
\begin{remark}
The same result holds for $\kappa =8$ as it was showed in \cite{lawler2011conformal} using a different approach.
\end{remark}

\subsection{Introduction to Quasi-sure Stochastic Analysis through Aggregation}

In this section, we introduce the Quasi-Sure Stochastic Analysis through Aggregation, following \cite{soner2011quasi}. We refer the reader to \cite{soner2011quasi} and \cite{denis2006theoretical} for further information. In \cite{soner2011quasi} the interest is to develop stochastic analysis simultaneously under an uncountable family of probability measures with certain properties (that are not dominated by a single probability measure) driven by applications, between others, in Mathematical Finance (for example, uncertain volatility models). In the context of $SLE_{\kappa}$ theory the parameter $\kappa$ will play the role of the uncertain volatility in our analysis. The results in  \cite{soner2011quasi} extend the theory that one naturally has for a fixed probability measure to define Stochastic Analysis simultaneously for a family of probability measures using the notion of \textit{aggregation} that we define in the following.

Let us consider the probability space $(\Omega, \mathbb{F}^B, \mathbb{P})$ where $\Omega=C(\mathbb{R}_{+}, \mathbb{R})$ and let $\mathbb{F}=\mathbb{F}^B$ be the filtration generated by the canonical process $B$.  For example, if $\mathbb{P}$ is the Wiener measure then the canonical process is a standard Brownian motion. Throughout our analysis we will consider a family of probability measures on this space indexed by a parameter, that will change the law of the canonical process under each of them.
We recall from \cite{soner2011quasi} that a probability measure $\mathbb P$ is a local martingale measure if the canonical process $B$ is a local martingale under $\mathbb P$. It is proved in \cite{karandikar1995pathwise} that there exists a progresively measurable process denoted as $\int_0^tB_sdB_s$ which coincides with the It\^o integral $\mathbb P$ -a.s. for all local martingale measures $\mathbb P$. In particular, this provides a pathwise definition of 
$$\langle B\rangle _t:=B_t^2-2\int_0^tB_sdB_s$$ and
$$\hat{a}_t:= {\limsup}_{\epsilon \to 0}\frac{1}{\epsilon}[\langle B \rangle_t- \langle B \rangle_{t-\epsilon}].$$

We first introduce as in \cite{soner2011quasi} the following notions. 
\begin{definition}
Let $\bar{\mathcal{P}}_W$ is the set of all local martingale measures $\mathbb P$ such that $\mathbb P$-a.s. $\langle B \rangle _t$ is absolutely continuous in $t$ and $\hat{a}$ takes values in $\mathbb R_+$.
\end{definition}
Let us fix $\mathcal{P} \subset \bar{\mathcal{P}}_W $, an arbitrary subset. We further introduce the notion of $\textit{capacity}$ that we will use throughout our analysis in the next sections.
\begin{definition}[Definition of capacity]
For each $f \in C_b(\Sigma)$- the set of bounded continuous functions on $\Sigma$, we put
$$cap(f)=\sup\{||f||_{L^2(\Sigma, P)} : \mathbb{P} \in \mathcal{P}\}.$$
For a measurable set $A$, we define $cap(A)=cap(I_A)$.
\end{definition}
Capacities naturally have applications in the theory of Risk Measures, see \cite{artzner1999coherent}, \cite{follmer2002convex}.

Next, we introduce the notion of polar set. 
\begin{definition}
 We say that a property holds $\mathcal{P}$-quasi-surely if it holds $\mathbb{P}$-a.s. for all the probability measures $\mathbb{P} \in \mathcal{P}$. We call a set $A$ polar if $cap(A)=0$, i.e. if $\mathbb{P}(A)=0$, for all $\mathbb{P} \in \mathcal{P}.$
\end{definition}

Let us denote $\mathcal{N}_\mathcal P:=\cap _{\mathbb P \in \mathcal P}\mathcal N^\mathbb P (\mathcal F _\infty),$ where $\mathcal{N}^\mathbb P$ is the collection of all $\mathbb{P}$-nullsets in $\mathcal F _\infty$. We use the following universal filtration $\mathfrak F^\mathcal P$ for the mutually singular measures $\{ \mathbb{P}, \mathbb{P} \in \mathcal P \}$.
$$\mathfrak F^\mathcal P :=\{ \mathcal F_t^\mathcal P\}_{t \geq 0}$$
where $$  \mathcal F_t^\mathcal P:= \cap_{\mathbb P \in \mathcal P}\left(\mathcal F_t^\mathbb P \vee \mathcal N_\mathcal P\right)\,.$$
The next definition introduces the notion of \textit{aggregator} that is fundamental in our analysis as it will allow us to describe objects in $SLE_{\kappa}$ theory simultaneously, for all the values $\kappa \in \mathcal{K}\cap \mathbb{R}_+ \setminus ([0, \epsilon) \cup \{8\})$, for any $\epsilon>0$, with $\mathcal{K} \subset \mathbb{R}_{+}$ a nontrivial compact interval.
\begin{definition}\label{aggregator}
Let $ \mathcal P \subset \bar{\mathcal{P}}_W $. Let $\{ X^\mathbb P, \mathbb P \in \mathcal P \}$ be a family of $\mathfrak F ^\mathcal P$ progressively measurable processes. An $\mathfrak F ^\mathcal P$ progressively measurable process $X$ is called a $\mathcal P$-aggregator of the family $\{ X^\mathbb P, \mathbb P \in \mathcal P\}$, if $X=X^\mathbb P$ , $\mathbb P$-a.s. for every $\mathbb P \in \mathcal P$.
\end{definition}
In order to assure the existence of a unique aggregator, one should have checked a consistency condition as well a separability assumption, that we discuss in the next section. In \cite{denis2006theoretical} the Black-Scholes model with uncertain volatility, i.e. with a volatility $\sigma \in [\sigma_m, \sigma_M]$, for constants $\sigma_m$ and $\sigma_M$, as a fundamental motivating example for the study of the quasi-sure stochastic analysis through aggregation. 
\subsection{The universal Brownian motion}
In this section, we introduce the notion of Universal Brownian motion as in \cite{soner2011quasi}, that is an example of an aggregator.\\
We first introduce the required tools to define this notion.  We refer the reader to \cite{soner2011quasi} for more details. 
Let
$$\bar{\mathcal{A}}:= \{ a : \mathbb{R}_+ \to \mathbb{R}_+ |\hspace{2mm} \mathbb F-\text{progresively measurable and} \int_0^t|a_s|ds <+\infty, \forall t \geq 0\}\,.$$

For a given $\mathbb P \in \bar {\mathcal P}_W$, let 

$$ \bar{\mathcal{A}}_W(\mathbb P):= \{ a \in \bar{\mathcal{A}}: a=\hat{a}, \mathbb P-a.s.\}$$

Recall that $\hat{a}$ is the density of the quadratic variation of $\langle B\rangle$ (where $B$ is the canonical process under the Wiener measure on path space) and is defined point-wise. We define 

$$\bar{\mathcal{A}}_W := \cup _{\mathbb P \in \bar{\mathcal{P}}_W} \bar{\mathcal{A}}_W(\mathbb P)$$

In order to construct a measure with a given quadratic variation $a \in \bar{\mathcal{A}}$ as in \cite{soner2011quasi}, we consider the weak solutions of the following stochastic differential equation 
\begin{equation}\label{eq4.4.}
dX_t=a_t^{1/2}(X)dB_t, \hspace{5mm} \mathbb{P}_0\text{-a.s.}
\end{equation}
Furthermore, if the equation \eqref{eq4.4.} has weak uniqueness, let $\mathbb{P}_a \in \bar{\mathcal{P}}_W$ be the unique solution of \eqref{eq4.4.} with initial condition $\mathbb{P}_a(B_0=0)=1,$ and we define 
$$\mathcal{A}_W:=\{a \in \bar{A}_W: \eqref{eq4.4.} \hspace{1mm} \text{has weak uniqueness} \}$$ 
$$\mathcal{P}_W:=\{\mathbb{P}_a, a \in \mathcal{A}_{W}\}$$
Let us fix a subset $\mathcal{A} \subset \mathcal{A}_W$. 
We further denote $$\mathcal{P}=\{\mathbb P_{a}, a \in \mathcal{A}\}.$$ We note that in the previous section, we discussed the case when $\mathcal{P}$ is an arbitrary subset of the set of measures, that in the context of this section takes a concrete form. 

Let us define for any $a,b \in \mathcal{A}$, the disagreement time $$\theta^{a,b}:=\inf \{ t \geq 0: \int_0^t a_sds \neq \int_0^t b_sds\}.$$

\begin{definition}
A subset $\mathcal{A}_0 \subset \mathcal{A}_W$ is called a generating class of diffusion coefficients if

\begin{itemize}

\item $\mathcal{A}_0$ satisfies the concatenation property  $a\bold{1}_{[0,t)}+b\bold{1}_{[t, \infty)} \in \mathcal{A}_0$, for $a, b, \in \mathcal{A}_0, t \geq 0.$ 

\item $\mathcal{A}_0$ has constant disagreement times: for all $a, b \in \mathcal{A}_0$, $\theta^{a,b}$ is constant. 

\end{itemize}
\end{definition}

\begin{definition}[Separability assumption]
Let $\mathcal{T}$ be the set of all $\mathbb{F}$-stopping times taking values in $\mathbb{R}_+ \cup \{\infty\}$. We say $\mathcal A$ is a separable class of diffusion coefficients generated by $\mathcal{A}_0$ if $\mathcal{A}_0 \subset \mathcal{A}_W$ is generated by a class of diffusion coefficients and $\mathcal{A}$ consists of all processes $a$ of the form
$$ a=\sum_0^{\infty}\sum_{i=1}^{\infty}a_i^n\bold{1}_{E_i^n}\bold{1}_{[\tau_n, \tau_{n+1})}$$
where $(a_i^n)_{i,n} \subset\mathcal{A}_0$, $(\tau_n)_n \subset \mathcal{T}$ is non-decreasing with $\tau_0=0$.
\begin{itemize}
\item We have that $ \inf\{n: \tau_n=\infty\}<\infty$ and $\tau_n <\tau_{n+1}$ whenever $\tau_n <\infty$ and each $\tau_n$ takes at most countably many values.

\item For each $n$ $\{E_i^n, i\geq 1\} \subset \mathcal{F}_{\tau_n}$ forms a partition of $\Omega$.
\end{itemize}
\end{definition}


A fundamental result that is proved in \cite{soner2011quasi} is the following theorem that assures that if the conditions of separability and consistency are satisfied, then one has a unique aggregator. 

\begin{theorem}[Theorem $5.1$ of \cite{soner2011quasi}] For $\mathcal{A}$ a separable class of diffusion coefficients generated by $\mathcal{A}_0$, let $\{ X^a, \hspace{2mm} a \in \mathcal{A} \}$ be a family of $\mathfrak F ^\mathcal P$-progressively measurable processes. Then there exists a unique ($\mathcal{P}$-q.s.) $\mathcal{P}$-aggregator $X$ if and only if $\{X^a, \hspace{2mm} a \in \mathcal{A}\}$ satisfies the consistency condition
$X^a=X^b$, $\mathbb{P}^a$ almost surely on $[0, \theta^{a, b})$ for any $a \in \mathcal{A}_0$ and $b \in \mathcal{A}$. 
\end{theorem}

As an application of this result, one can construct an aggregator for the Brownian motion.  For this, let us consider a standard Brownian motion $B_t$ (the canonical process under the Wiener measures $\mathbb{P}_0$ as in \cite{soner2011quasi}). For any $\mathbb{P}_a \in \mathcal{P}_W$ and $a \in \bar{\mathcal A}_W(\mathbb P)$  by L\'evy's characterization, we obtain that the following It\^o stochastic integral under $\mathbb{P}_a$ is a $\mathbb {P}_a$- Brownian motion

$$W_t^{\mathbb P_a}:=\int_0^t a_s^{-1/2}dB_s$$

For $\mathcal{A}$ satisfying the consistency condition, the family $\{ W^{\mathbb{P}_a}, a \in \mathcal{A}\}$ admits a unique $\mathcal{P}$-aggregator $W_t$ (see \cite{soner2011quasi}). 
The aggregator is the following stochastic integral  $$W^{\mathcal{P}}_t=\int_0^t a_s^{-1/2}dB_s$$ that is defined quasi-surely, i.e. simultaneously for all the measures $\mathbb{P}_a \in \mathcal{P} $ satisfying the conditions. The construction of such an object is done in Corollary $5.5$ in \cite{soner2011quasi}. Since $ W_t^{\mathcal{P}}$ is a $\mathbb{P}_a$ Brownian motion for every $a \in \mathcal{A}$, we call $W_t^{\mathcal{P}}$- a universal Brownian motion.


When studying $SLE_{\kappa}$ theory, the natural process to be considered is $\sqrt{\kappa}B_t$ with $B_t $ a standard Brownian Motion.
Thus, we will work with $a_s^{1/2}dW^{\mathbb{P}_a}_t$, i.e. $a_s=\kappa$.
 Since in \cite{soner2011quasi}, the process $W_t^{\mathbb{P}_a}$ is defined for all $a$, and is a $\mathbb{P}_a$-standard BM, we will just modify its quadratic variation by constants. 
Thus, in our analysis, for any $\epsilon>0$,  we consider the family of measures 
\begin{align}\label{familyofmeasures}
\mathcal{P}_{\kappa}:= \{\mathbb{P}_{\kappa} :  \kappa \in \mathcal{K}\cap \mathbb{R}_+ \setminus ([0, \epsilon) \cup \{8\}) \}
\end{align}
obtained as the measures $\mathcal{P}$, i.e. as weak solutions to the equation \eqref{eq4.4.} with $a_s=\kappa$ with $\kappa$ in a compact non-trivial interval.
We use the notation $W^{\mathcal{P}_{\kappa}}_t$ to refer to the universal Brownian motion under the family of measures $\mathcal{P}_{\kappa}.$ We use the notation $W^{\mathbb{P}_{\kappa}}_t$ to refer to the $\mathbb{P}_{\kappa}$ -BM under the measure $\mathbb{P}_{\kappa}$.

 In \cite{denis2006theoretical}, it is studied the Black-Scholes model with uncertain volatility, i.e. with a volatility $\sigma \in [\sigma_m, \sigma_M]$ for $\sigma_m$ and $\sigma_M$ some constants. Then, the $\sqrt{\kappa}W_t^{\mathbb{P}_{\kappa}}$ is a closely related process that one can define quasi-surely when varying the parameter $\kappa$ in a certain interval (see also Example  $4.5$ in \cite{soner2011quasi}).



A fundamental result that we use is the aggregate solution to stochastic differential equations.
In the paper, they show how to solve a stochastic differential equation simultaneously under all the measures $\mathbb{P} \in \mathcal{P}$. Specifically, they prove the following result:

\begin{proposition}[Proposition $6.10$ of \cite{soner2011quasi}]\label{propaggregat}
Let $\mathcal{T}$ be the set of all $\mathbb{F}$-stopping times taking values in $\mathbb{R}_+\cup\{\infty\}.$ Let $\mathcal{A}$ satisfy the consistency assumption. Assume that for every $\mathbb{P} \in \mathcal P$ and $\tau \in \mathcal{T},$  the equation
$$X_t=X_0+\int_0^tb(X_s)ds+\int_0^t\sigma_s(X_s)dB_s, $$
has a unique $\mathbb{F}^\mathbb P$ progressively measurable strong solution on the interval $[0, \tau]$. Then there exists $\mathcal P$-q.s. aggregated solution (see Def. \ref{aggregator}) to the equation above, i.e. 
$$X_t=X_0+\int_0^tb(X_s)ds+\int_0^t\sigma_s(X_s)dB_s, \hspace{3mm} t \geq 0$$ has solution simultaneuously under all prrobability measures $\mathbb{P}_{\kappa}.$
\end{proposition}
 
In the next section, we use this result for the stochastic differential equation corresponding to $\tilde{K}_s$ that is used in the construction of the $SLE_{\kappa}$ trace.

\section{Heuristics of the quasi-sure construction of the $SLE_{\kappa}$ traces}
Following the parametrization in [Chapter 7, \cite{lawler2008conformally}] we set $\hat{g}_t(z):=\frac{g_t(\sqrt{\kappa}z)}{\sqrt{\kappa}}$ for the maps $g_t(z)$ satisfying the forward Loewner differential equation. Thus, the maps $\hat{g}_t(z)$ satisfy the differential equation
$\partial_t\hat{g}_t(z)=\frac{2/\kappa}{\hat{g}_t(z)-B_t}$.
We work with this paramatrization of the dynamics in order to keep the exposition of the method in line with the approach from [Chapter 7. \cite{lawler2008conformally}].  This formulation is equivalent with the one in which the parameter $\kappa$ is kept in $\sqrt{\kappa}B_t$. To be precise, the Proposition \ref{propaggregat} applies for the process $\sqrt{\kappa}B_t$ and then we reparametrize. We avoid doing the details and work directly with this parametrization in order to keep the exposure neat. 

The main idea is to consider the construction of the aggregated solution to SDE as in Proposition \ref{propaggregat} and applied to the SDE corresponding to the process $\tilde{K}_s$. Furthermore,  we express the derivative of the map $\tilde{h}_t(z_0)$ using the aggregated solution of an SDE and use the lemmas in \cite{lawler2011conformal} to obtain the quasi-sure existence of the $SLE_{\kappa}$ trace. 

We have
$$|\tilde{h}_t'(z_0)|=e^{-\frac{2}{\kappa}t}\exp \left(\frac{4}{\kappa} \int_0^t \frac{\tilde{K}^2_s+1}{\tilde{K}^2_s-1}ds \right)=e^{\frac{2}{\kappa}t}\exp \left(\frac{4}{\kappa}\int_0^t \tilde{N}_s ds \right)\,,$$
where
$$ d \tilde{N}_s=(1-\tilde{N}_s)[-4(\frac{2}{\kappa}+1)\tilde{N}_s+1]ds+2\sqrt{\tilde{N}_s}(1-\tilde{N}_s)d\tilde{B}_s.$$
The SDE for $\tilde{K}_t$ is
$$d\tilde{K}_t=\frac{4}{\kappa}\tilde{K}_tdt+\sqrt{1+\tilde{K}_t^2}d\tilde{B}_t,$$
where $\tilde{B}_t$ is a standard Brownian motion. 


 
%

%
Checking the conditions of Yamada-Watanabe Theorem  for the SDE $\tilde{K}_t$ (see \cite{revuz2013continuous}), we obtain that this stochastic differential equation has a unique strong solution.

 Once we have a unique notion of strong solution for the SDE for $K_t$,  $\mathbb{P}$-a.s., we can construct an aggregated solution for this SDE using Proposition \ref{propaggregat} and then express
 
$$|\tilde{h}_t'(z_0)|=e^{-2t/\kappa}\exp \left(\frac{4}{\kappa} \int_0^t \frac{\tilde{K}^2_s+1}{\tilde{K}^2_s-1}ds \right).$$

 
Using the aggregated solution for the SDE $\tilde{K}_t$ we construct simultaneously the $SLE_{\kappa}$ trace for all parameters $\kappa \in \mathcal{K}\cap \mathbb{R}_+ \setminus ([0, \epsilon) \cup \{8\})$, for any $\epsilon>0$, where $\mathcal{K}$ is a nontrivial compact interval of $\mathbb{R}_+$ by relating the aggregated solution for $\tilde{K}_t$ with the derivative of the backward $SLE$ map $h_t(z)$. Then, having the quantitative estimate simultaneously for all the probability measures $\mathbb{P}_{\kappa}$ in $\mathcal{P}_{\kappa}$ will allow us to construct the $SLE_{\kappa}$ traces quasi-surely.



\color{black}

 \begin{section}{The a.s. existence of the $SLE_{\kappa}$ trace for fixed $\kappa \in \mathbb{R}_+ \setminus \{8\}$}

For the convenience of the reader, we recall the a.s. construction of the $SLE_{\kappa}$ trace for any fixed $\kappa \neq 8$ from \cite{schramm2005basic}. The elements used in the proof of the existence of the $SLE_{\kappa}$ trace for fixed $\kappa$, a.s., are listed in the following. For more details, we refer the reader to \cite{schramm2005basic} and \cite{lawler2008conformally}.

%
%

\begin{subsection}{Estimates for the mean of the derivative for a fixed $\kappa$}

Considering the real and the imaginary part of the backward $SLE$, we have that
\begin{equation}
dX_t =\frac{-2 X_t}{X_t^2+Y_t^2}dt-\sqrt{\kappa}dB_t\,, \hspace{5mm} dY_t=\frac{2Y_t}{X_t^2+Y_t^2}dt\,,
\end{equation}
We consider the time change $\sigma(t)=X_t^2+Y_t^2\,,$ $t=\int_0^{\sigma(t)}\frac{ds}{X_s^2+Y_s^2}\,.$ With the new time, we define the random variables $\tilde{Z}_t=Z_{\sigma (t)} \,,$ $\tilde{X}_t=X_{\sigma(t)}\,,$ and $\tilde{Y}_t=Y_{\sigma(t)}\,.$

The first elements of the proof of the existence of the $SLE_{\kappa}$ trace are the following proposition and the corollary of it. 
\begin{proposition}[Proposition 7.2 in \cite{lawler2011conformal}] \label{prop4.1}
Let $r, b$ such that 
$$r^2-\left(\frac{4}{\kappa}+1\right)r+\frac{2}{\kappa}b=0\,, $$
then
$$ M_t\;:=\; \tilde{Y}_t^{b-(r\kappa/2)}(|\tilde{Z}_t|/\tilde{Y}_t)^{2r}|h_t'(z_0)|^{b}\,,$$
is a martingale. Moreover,
$$\mathbb{P}(|\tilde{h}_{t}'(z_0)| \geq \lambda) \leq \lambda^{-b}(|z_0|/y_0)^{2r}e^{t(r-2b/\kappa)}\,. $$

\end{proposition}
%
%

\begin{corollary}[Corollary $7.3$ in \cite{lawler2011conformal}] \label{coro}
For every $0 \leq r \leq \frac{4}{\kappa}+1\,,$ there is a finite $c=c(\kappa,r)$ such that for all $0 \leq t \leq 1$, $0 \leq y_0 \leq 1\,,$ $e \leq \lambda \leq y_0^{-1}\,,$ we have that

$$\mathbb{P}(|h_{t}'(z_0)| \geq \lambda) \leq \lambda^{-b}(|z_0|/y_0)^{2r}\delta(y_0, \lambda)\,, $$
where $b=\frac{[(\frac{4}{\kappa}+1)r-r^2]\kappa}{2} \geq 0$ and

\[
    \delta(y_0, \lambda)= 
\begin{cases}
    \lambda^{(r\kappa/2)-b},& \text{if }   r < \frac{2b}{\kappa}\,,\\
    -\log (\lambda y_0), & \text{if} \hspace{1mm} r = \frac{2b}{\kappa}\,,\\
    y_0^{b-(r\kappa/2)}, & \text{if }  r > \frac{2b}{\kappa}\,.
\end{cases}
\]
\end{corollary}
%
%

\subsection{Existence of the trace for fixed $\kappa$}

\begin{proposition}[Proposition $4.33$ in \cite{lawler2011conformal}]
Suppose that $g_t$ is a Loewner chain with driving function $U_t$ and assume that there exist a sequence of positive numbers $r_j \to 0$ and a constant $c$ such that
\begin{align*}
|\hat{f}_{k2^{-2j}}'(2^{-j}i)| \leq 2^j r_j\,, k=0,1, \ldots, 2^{2j}-1\,,\\
|U_{t+s}-U_t| \leq c\sqrt{j}2^{-j}\,, 0 \leq t \leq 1, 0 \leq s \leq 2^{-2j}\,.
\end{align*}
and
\begin{align*}
\lim_{j \to \infty} \sqrt{j}/ \log r_j =0\,. 
\end{align*}
Then $V(y, t) := \hat{f}_{t}(iy)$ is continuous on $[0,1] \times [0,1] \,.$
\end{proposition}
Combining the previous results, one obtains Theorem \ref{Rohdeschramm} that we repeat for convenience. 
\begin{theorem}[Theorem $4.1$ of \cite{schramm2005basic}]
Let $(K_t)_{t \geq 0}$ be a $SLE_{\kappa}$ for $\kappa \neq 8\,.$  Then, $\hat{g}_{t}^{-1}(z)= g_{t}^{-1}(z+\sqrt{\kappa}B_t) :\mathbb{H} \mapsto H_t$ extends continuously to $\bar{\mathbb{H}}$ for all $t \geq 0,$ almost surely. Moreover, $\gamma_t$ is continuous and generates $(K_t)_{t \geq 0}$ almost surely.
\end{theorem}

\end{subsection}

\color{black}

\end{section}

\section{Quasi-sure existence of the $SLE_{\kappa}$ trace -defining the $SLE_{\kappa}$ trace simultaneously for all $\kappa \in \mathcal{K}\cap \mathbb{R}_+ \setminus ([0, \epsilon) \cup \{8\})$}

In this section, we construct the $SLE_{\kappa}$ quasi-surely, i.e. we construct the $SLE_{\kappa}$ traces simultaneously for the family of measures $\mathbb{P}_{\kappa}$, for $\kappa \in \mathcal{K}\cap \mathbb{R}_+ \setminus ([0, \epsilon) \cup \{8\}).$ We obtain in this manner an \textit{aggregator} of the SLE traces.

\subsection{Estimates on the moments of the derivatives for many $\kappa$ using aggregation of solutions of a SDE}

We consider the family of measures $\mathcal{P}_{\kappa}$ as in \eqref{familyofmeasures} and the aggregator of the Brownian motion under this family of measures $W_t^{\mathcal{P}_{\kappa}}$ described in the previous section. In order to define the $SLE_{\kappa}$ traces simultaneously for the family of measures $\mathcal{P}_{\kappa}$ we study the Loewner differential equation driven by a simple modification of the aggregator $W_t^{\mathcal{P}_{\kappa}}$. Specifically, we consider $\sqrt{\kappa}W_t^{\mathbb{P}_{\kappa}}$ as a driver for the backward Loewner differential equation, for all the measures $\mathbb{P}_{\kappa}$, with $\kappa\in \mathcal{K}\cap \mathbb{R}_+ \setminus ([0, \epsilon) \cup \{8\})$, for any $\epsilon>0$. One can think about this also as the forward Loewner differential equation driven by the canonical process under the family of measures $\mathcal{P}_{\kappa}$.

Investigating the real and the imaginary part of the backward $SLE$, we have that under each $\mathbb{P}_{\kappa} \in \mathcal{P}_{\kappa}$, we have that 
\begin{equation}
dX_t =\frac{-2 X_t}{X_t^2+Y_t^2}dt-\sqrt{\kappa}dW_t^{\mathbb{P}_{\kappa}}\,, \hspace{5mm} dY_t=\frac{2Y_t}{X_t^2+Y_t^2}dt\,,
\end{equation}
 We recall from the heuristics that we follow the parametrization in [Chapter 7, \cite{lawler2008conformally}] and we set $\hat{g}_t(z):=\frac{g_t(\sqrt{\kappa}z)}{\sqrt{\kappa}}$ for the maps $g_t(z)$ satisfying the forward Loewner differential equation. Thus, the maps $\hat{g}_t(z)$ satisfy the differential equation
$\partial_t\hat{g}_t(z)=\frac{2/\kappa}{\hat{g}_t(z)-B_t}$. 
We recall also that this choice is for the convenience of the analysis and does not change the aggregation result.
We consider the time change $\sigma(t)=X_t^2+Y_t^2\,,$ $t=\int_0^{\sigma(t)}\frac{ds}{X_s^2+Y_s^2}\,.$ With the new time, we define the random variables $\tilde{Z}_t=Z_{\sigma (t)} \,,$ $\tilde{X}_t=X_{\sigma(t)}\,,$ and $\tilde{Y}_t=Y_{\sigma(t)}\,.$
Furthermore, let us consider $\sqrt{\kappa}\tilde{W}_t^{\mathbb{P}_{\kappa}}:=\frac{\sqrt{\kappa}dW^{\mathbb{P}_{\kappa}}_t}{\sqrt{X_t^2+Y_t^2}}$.  Using the L\'evy's characterization of Brownian motion (for every $\mathbb{P}_{\kappa})$ we deduce that the random time changed Brownian motion is also a Brownian motion (in the random time defined above) for all $\mathbb{P}_{\kappa}$.


Using Yamada-Watanabe Theorem (see \cite{revuz2013continuous}), the following SDE for $\tilde{W}_t$ being a standard Brownian motion has a unique strong solution 
$$d\tilde{K}_t=\frac{4}{\kappa}\tilde{K}_tdt+\sqrt{1+\tilde{K}_t^2}d\tilde{W}_t. $$
Next, we use Proposition \ref{propaggregat} in order to obtain the aggregated solution of the SDE for $\tilde{K_t}$:
$$d\tilde{K}_t=\frac{4}{\kappa}\tilde{K}_tdt+\sqrt{1+\tilde{K}_t^2}d\tilde{W}_t^{\mathcal{P}_{\kappa}},$$ 
where $\tilde{W}_t^{\mathcal{P}_{\kappa}}$ is the aggregator of the Brownian motion for the family of measures $\mathbb{P}_{\kappa}$.
Specifically, for any measure $\mathbb{P}_{\kappa}$ we have 
\begin{equation}\label{SDE K_t}
d\tilde{K}_t=\frac{4}{\kappa}\tilde{K}_tdt+\sqrt{1+\tilde{K}_t^2}d\tilde{W}_t^{\mathbb{P}_{\kappa}},
\end{equation}
where $\tilde{W}_t^{\mathbb{P}_{\kappa}}$ is a standard $\mathbb{P}_{\kappa}$-Brownian motion.

In order to prove similar estimates that were obtained  in the previous section for fixed $\kappa$ simultaneously for all $\kappa$, we use the aggregated solution and relate it with the derivative of the map, via 
\begin{equation}\label{SDE to map}
|\tilde{h}_t'(z_0)|=e^{-\frac{2t}{\kappa}}\exp \left( \frac{4}{\kappa} \int_0^t \frac{\tilde{K}^2_s+1}{\tilde{K}^2_s-1}ds \right).
\end{equation}
Using the aggregated solution, we obtain a version of Proposition \ref{7.2} using the family of measures $\mathbb{P}_{\kappa}$. In this manner we can construct the trace by obtaining an estimate for the derivative of the conformal maps $h_t(z)$ similar to the one in Proposition\label{7.2} simultaneously for all $\kappa \in \mathcal{K}\cap \mathbb{R}_+ \setminus ([0, \epsilon) \cup \{8\})$ using the aggregated solution. In this manner we obtain the q.s. existence of SLE traces simultaneously for all $\kappa \in \mathcal{K}\cap \mathbb{R}_+ \setminus ([0, \epsilon) \cup \{8\})$, for any $\epsilon>0$.

First, we prove a version of Proposition \ref{prop4.1} for the family of measures $\mathbb{P}_{\kappa} \in \mathcal{P}_{\kappa}.$

\begin{proposition} \label{7.2qs}
Let $r, b$ such that 
$$r^2-\left(\frac{4}{\kappa}+1\right)r+\frac{2b}{\kappa}=0\,, $$
then
$$ M_t\;:=\; \tilde{Y}_t^{b-(\frac{2r}{\kappa})}(|\tilde{Z}_t|/\tilde{Y}_t)^{2r}|h_t'(z_0)|^{b}\,,$$
is a martingale under any measure $\mathbb{P}_{\kappa} \in \mathcal{P}_{\kappa}$ as in \eqref{familyofmeasures}. Moreover, for any measure $\mathbb{P}_{\kappa}$, we have
$$\mathbb{P}_{\kappa}(|\tilde{h}_{t}'(z_0)| \geq \lambda) \leq \lambda^{-b}(|z_0|/y_0)^{2r}e^{t(r-\frac{2b}{\kappa})}\,. $$

\end{proposition}
\begin{proof}
By applying the chain rule for the function $L_t= \log h'_t(z_0)$, we obtain that
$L_t=-\int_0^t\frac{2/\kappa}{Z_s^2}ds\,,$ and in particular,
$|\tilde{h}_t'(z_0)|=\exp \left(\frac{2}{\kappa}\int_0^t \frac{\tilde{Y}_s^2-\tilde{X}_s^2}{\tilde{X_s}^2+\tilde{Y_s}^2} ds \right)$.
Moreover, for fixed $\kappa$ we have $\tilde{K}_t=\frac{\tilde{X}_t^2}{\tilde{Y}_t^2}$ and $\tilde{N}_t=\frac{\tilde{K}_t}{1+ \tilde{K}_t}.$ Then, for fixed $\kappa$, we obtain that 
$$|\tilde{h}_t'(z_0)|=e^{-\frac{2t}{\kappa}}\exp \left(\frac{4}{\kappa}\int_0^t \tilde{N}_s ds \right)\,.$$ Next, we use the aggregated solution for $\tilde{K}_t$ and we obtain via $\tilde{N}_t=\frac{\tilde{K}_t}{1+ \tilde{K}_t}$ an aggregator for $\tilde{N}_t$ as well. Next, we prove similar estimates as one obtains for fixed $\kappa$, using the
aggregated solution $\tilde{N}_s$. 


In the $\sigma(t)$ time parametrization, we have that $d\tilde{Y}_t=-\frac{2}{\kappa}\tilde{Y}_t dt\,,$ so in this time parametrization $\tilde{Y}_t$ grows deterministically  $\tilde{Y}_t=\tilde{Y_0}e^{\frac{2}{\kappa}t}\,.$
At this moment, we can rephrase the formula for $M_t$ as 
$$M_t=y_0^{b-(\frac{\kappa r}{2})}e^{-rt}(1-\tilde{N_t})^{-r} \exp \left(\frac{4b}{\kappa} \int_0^t \tilde{N}_s ds\right)\,. $$
and by applying It\^o formula, we obtain that
$$dM_t=2r\sqrt{\tilde{N_t}}M_t d\tilde{W}^{\mathbb{P}_{\kappa}}_t\,, $$
where $d\tilde{W}^{\mathbb{P}_{\kappa}}_t=\int_0^{\sigma(t)} \frac{1}{\sqrt{X_t^2+Y_t^2}}dW^{\mathbb{P}_{\kappa}}_t$ is the Brownian motion that we obtain in the time reparametrization.
This shows that $M_t$ is a martingale, hence
$$\mathbb{E}_{\kappa}[M_t]=\mathbb{E}_{\kappa}[M_0]=y_0^{b-(r\kappa/2)}(|z_0|/y_0)^{2r}\,. $$
Note that since for $r \geq 0\,,$ $(|\tilde{Z}_t|/\tilde{Y_t})^{2r} \geq 1\,,$ then  by Markov inequality, we have that
$$\mathbb{P}_{\kappa}(|\tilde{h}_{t}'(z_0)| \geq \lambda) \leq \lambda^{-b}(|z_0|/y_0)^{2r}e^{t(r-\frac{2b}{\kappa})}\,. $$

\end{proof}

Moreover, we obtain a version of the Corollary \ref{coro} under the family of measures $\mathbb{P}_{\kappa}.$

\begin{corollary}\label{coroqs}
For every $0 \leq r \leq \frac{4}{\kappa}+1\,,$ there is a finite $c=c(\kappa, r)$ such that for all $0 \leq t \leq 1$, $0 \leq y_0 \leq 1\,,$ $e \leq \lambda \leq y_0^{-1}\,,$ for any $\mathbb{P}_{\kappa} \in \mathcal{P}_{\kappa}$ as in \eqref{familyofmeasures}, we have that

$$\mathbb{P}_{\kappa}(|h_{t}'(z_0)| \geq \lambda) \leq \lambda^{-b}(|z_0|/y_0)^{2r}\delta(y_0, \lambda)\,, $$
where $b=\frac{[(\frac{4}{\kappa}+1)r-r^2]\kappa}{2} \geq 0$ and

\[
    \delta(y_0, \lambda)= 
\begin{cases}
    \lambda^{(r\kappa/2)-b},& \text{if }   r < \frac{2b}{\kappa}\,,\\
    -\log (\lambda y_0), & \text{if} \hspace{1mm} r = \frac{2b}{\kappa}\,,\\
    y_0^{b-(r\kappa/2)}, & \text{if }  r > \frac{2b}{\kappa}\,.
\end{cases}
\]
\end{corollary}
\begin{proof}
From $dY_t=\frac{2Y_t}{X_t^2+Y_t^2}dt\,,$ we obtain that $dY_t \leq \frac{2/\kappa}{Y_t}dt\,,$ and hence we obtain in the following that $Y_t \leq \sqrt{\frac{4}{\kappa}t+y_0^2}\leq \sqrt{\frac{4}{\kappa}+1}\,.$ In the last inequality, we used that $t \leq 1$ and $y_0 \leq 1\,.$ Using the exponential growth of $Y_t$ in this time reparametrization, we obtain that $\tilde{Y_t}=\sqrt{\frac{4}{\kappa}+1}$ at time $T=\frac{\log\sqrt{\frac{4}{\kappa}+1}-\log y_0}{2/\kappa}\,.$

Therefore,
$$\mathbb{P}_{\kappa}(|h_{t}'(z_0)| \geq \lambda ) \leq \mathbb{P}_{\kappa}(\sup_{0 \leq s \leq T}|\tilde{h}_s'(z_0)| \geq \lambda)\,.$$

Using that $|\tilde{h}_t'(z_0)|=e^{-\frac{2t}{\kappa}}\exp \left(\frac{4}{\kappa} \int_0^t \tilde{N}_s ds \right)$ we obtain that $|\tilde{h}'_{t+s}(z_0)| \leq e^{2s/\kappa}|\tilde{h}'_t(z_0)|\,.$ So by addition of the probabilities, we have that
\begin{align*}
\mathbb{P}_{\kappa}(\sup_{0 \leq t \leq T}|\tilde{h}'_t(z_0)| \geq e^{2/\kappa}\lambda) \leq \sum\limits_{j=0}^{[T]}\mathbb{P}_{\kappa}(|\tilde{h}'_j(z_0)|  \geq \lambda)\,.
\end{align*} 
Using the Schwarz-Pick Theorem for the upper half-plane we obtain that $|\tilde{h}_{t}(z_0)|\leq \text{Im} \tilde{h}_t'(z_0)/y_0=e^{2t/\kappa}\,.$ This gives a lower bound for the $t$ that we are summing over and we obtain that
via the Proposition \ref{7.2} that 
\begin{align*}
\mathbb{P}_{\kappa}(\sup_{0 \leq t \leq T}|\tilde{h}'_t(z_0)| \geq e^{2/\kappa}\lambda) &\leq \sum\limits_{(\kappa/2)\log \lambda \leq j \leq T} \mathbb{P}_{\kappa}(|\tilde{h}_j'(z_0)|\geq \lambda)\\
&\leq \lambda^{-b}(|z_0|/y_0)^{2r}\sum\limits_{ (\kappa/2)\log \lambda \leq j \leq T} e^{j(r-ab)} \\
&\leq c\lambda^{-b} (|z_0|/y_0)^{2r}\delta(y_0, \lambda)\,.
\end{align*}
\end{proof}

In order to prove the result, we need the following Lemma, that we recall from the introduction for the convenience of the reader. 
\begin{lemma}[Lemma 5.5 of \cite{kemppainen2017schramm}]\label{samed}
Let $h_t(z)$ be the solution to the backward Loewner differential equation with driving function $\sqrt{\kappa}B_t$ and let $f_t(z)$ be the solution of the partial differential equation version of the Loewner differential equation with the same driver. Then, for any $t \in \mathbb{R}_+$, the function $z \to f_t(z+\sqrt{\kappa}B_t)-\sqrt{\kappa}B_t$ and $z \to h_t(z)$ have the same distribution.
\end{lemma}
We note that the previous lemma is also true when one changes the measures $\mathbb{P}_{\kappa}$, for $\kappa \in \mathcal{K}\cap \mathbb{R}_+ \setminus ([0, \epsilon) \cup \{8\})$, for any $\epsilon>0$.
We combine the previous results, to obtain the quasi-sure existence of the $SLE_{\kappa}$ trace. 
\begin{theorem}
Let $\epsilon>0$. Let $\mathcal{K} \subset \mathbb{R}_+$ be a nontrivial compact interval and let us consider the family of probability measures $\mathcal{P}_{\kappa}$ as in \eqref{familyofmeasures}.  Then, for $\kappa \in \mathcal{K}\cap \mathbb{R}_+ \setminus ([0, \epsilon) \cup \{8\})$ the chordal $SLE_{\kappa}$ is quasi surely (i.e $\mathbb{P}_{\kappa}$-a.s. for all $\kappa \in \mathcal{K}\cap \mathbb{R}_+ \setminus ([0, \epsilon) \cup \{8\})$) generated by a path .
\end{theorem}
\begin{proof}
Using the scaling of the $SLE_{\kappa}\,,$ it suffices to prove the Theorem only for $t \in [0,1]\,.$
In order to prove the existence of the trace for fixed $\kappa$ a.s., according to the analysis presented in the previous section from \cite{lawler2008conformally} (see also \cite{schramm2005basic}) one needs to obtain a.s. the estimates 
\begin{align*}
|f'_{k2^{-2j}}(i2^{-j})| \leq \tilde{C}(\kappa, \omega)2^{j-\epsilon_1}\,, j=1,2,\ldots, k=0,1, \ldots, 2^{2j}\,,\\
\sqrt{\kappa}|W_t-W_s| \leq c_1(\kappa, \omega) \sqrt{\kappa}|t-s|^{1/2}|\log \sqrt{|t-s|}| \hspace{5mm}  0 \leq t \leq 1\,.
\end{align*}
In order to adapt the proof for fixed $\kappa$ of the existence of the $SLE_{\kappa}$ traces from the previous section in the Quasi-Sure Analysis through Aggregation setting, it suffices to show that q.s. there exists an $\epsilon_1>0$ and a random constant $C$ (that depends on the worst $\kappa$ in the sense of convergence of the probabilities series in the Borel-Cantelli argument, see below) such that
\begin{align*}
|f'_{k2^{-2j}}(i2^{-j})| \leq C2^{j-\epsilon_1}\,, j=1,2,\ldots, k=0,1, \ldots, 2^{2j}\,,\\
\sqrt{\kappa}|W_t^{\mathbb{P}_{\kappa}}-W_s^{\mathbb{P}_{\kappa}}| \leq c_1\sqrt{\kappa}|t-s|^{1/2}|\log \sqrt{|t-s|}| \hspace{5mm}  0 \leq t \leq 1\,.
\end{align*}

 The second inequality holds $\mathbb{P}_{\kappa}$-a.s. for every measure $\mathbb{P}_{\kappa} \in \mathcal{P}_{\kappa}$ as in \eqref{familyofmeasures} and is a consequence of the modulus of continuity for the Brownian motion. 
For the first inequality, we use the aggregated solution of the SDE as in \eqref{SDE K_t} and the relation \eqref{SDE to map} the previous section and we consider $r=\frac{2}{\kappa}+\frac{1}{4}< \frac{4}{\kappa}+1$ and 
$b=\frac{(1+\frac{4}{\kappa})r-r^2}{2/\kappa}=\frac{2}{\kappa}+1+\frac{3}{32/\kappa}\,,$ according to the Corollary \ref{coroqs}\,. Thus, we are in the regime $r <\frac{2b}{\kappa}\,,$ so by Corollary \ref{coroqs} we have the following $\kappa$ dependent bound
\begin{equation}\label{kdependentderivativeestimate}
\mathbb{P}_{\kappa}(|h'_t(i2^{-j})| \geq 2^{j-\epsilon_1}) \leq c2^{-j(2b-(2r/\kappa))(1-\epsilon_1)}
\end{equation}
In order to obtain the estimate for the family of functions $f_t(z)$, we use for each measure $\mathbb{P}_{\kappa}$ the Lemma \ref{samed}.

We repeat the optimization procedure for each $\kappa \in \mathcal{K}\cap \mathbb{R}_+ \setminus ([0, \epsilon) \cup \{8\})$ and obtain a control for  $c$ in $\kappa$ in the estimate \eqref{kdependentderivativeestimate}. We present parts of the analysis in the following. For more details on how to optimize over the parameters $r$ and $b$ for fixed $\kappa$ we refer the reader to the proof of the existence of the trace in \cite{kemppainen2017schramm}.
Following the analysis from Section $5.5$ of \cite{kemppainen2017schramm}, the parameter $b$ is expressed in terms of $r$ for fixed $\kappa$. In this setting, we have $b(r)=\frac{\kappa((1+4/\kappa)r-r^2)}{2}$. Then, one studies the quantity $\alpha(r)=2b(r)-\frac{r\kappa}{2}$ that is maximized by $r_0=\frac{1}{4}+\frac{2}{\kappa}$. Thus, $\alpha(r_0)=\kappa(1/4+2/\kappa)^2 \geq 2$ and $\alpha(r_0)=2$ if and only if $\kappa=8$. Then one obtains  $b=b_0=\frac{\kappa((1+4/\kappa)r_0-r_0^2)}{2}$. 
One  can express explicitly the constant $c$ in the estimate 
$\mathbb{P}_{\kappa}(\sup_{0 \leq t \leq T}|\tilde{h}'_t(z_0)| \geq e^{2/\kappa}\lambda) \leq c\lambda^{-b} (|z_0|/y_0)^{2r}\delta(y_0, \lambda)$ 
as an explicit function of $p_0$, $\kappa$ and $r_0$ in the case $ b_0-\kappa r_0/2 \leq 0$ i.e. $\kappa>8$ and similarly when $ b_0-\kappa r_0/2 \geq 0$, i.e. when $\kappa <8$. 
Moreover, one can control this expressions and give a bound for $c$ for all the $\kappa \in \mathcal{K}\cap \mathbb{R}_+ \setminus ([0, \epsilon) \cup \{8\})$, for any $\epsilon>0$.

Thus, in order to obtain the first inequality for all the measures $\mathbb{P}_{\kappa}$ in our interval of interest, we use the aggregated solution of the SDE as in \eqref{SDE K_t} and the relation \eqref{SDE to map} along with  Borel-Cantelli Lemma for capacities (see \cite{bouleau2010dirichlet}, \cite{denis2006theoretical}, \cite{songanother}) and Lemma \ref{samed}  to find $c $ (that can be chosen a continuous function of $\kappa$) and $\epsilon_1>0$ such that for all $0 \leq t \leq 1$

\begin{equation}\label{fixk}
\sup_{\kappa}\mathbb{P}_{\kappa}(|h'_t(i2^{-j})| \geq 2^{j-\epsilon_1}) \leq    \sup_{\kappa}c2^{-j(2b-(2r/\kappa))(1-\epsilon_1)}. 
\end{equation}

We obtain that $2b-(2r/\kappa) = 4/\kappa+1+\kappa/16 >2 $ provided that $2/\kappa \neq 1/4$ . So, we can apply Borel-Cantelli argument for capacities provided that $2/\kappa \neq 1/4$, i.e. $\kappa \neq 8$. We restrict to  $\kappa \in \mathcal{K}\cap \mathbb{R}_+ \setminus ([0, \epsilon) \cup \{8\})$, in order to obtain the control on the constant $c$ in \eqref{fixk}.
Thus, we have that
$$cap(|h'_t(i2^{-j})| \geq 2^{j-\epsilon_1}) \leq c2^{-(2+\epsilon_1)j} \,.$$
Let us consider the dyadic partition of the time interval $[0,1]$, $\mathcal{D}_{2n}=\{l2^{-2n}: l \in [0, 2^{2n}]\}$. Then,

\begin{equation}\label{finalest}
\sum_{n \in \mathbb{N}}\sum_{t \in \mathcal{D}_{2n}}cap(|h'_t(i2^{-j})| \geq 2^{j-\epsilon_1}) < \infty, 
\end{equation}
and we obtain the desired conclusion.
\end{proof}


\begin{remark}
In the previous result, we showed how one can aggregate the $SLE_{\kappa}$ curves for the family of measures $\mathcal{P}_{\kappa}$. For all the given measures $\mathbb{P}_{\kappa}$ one can obtain under all the measures $\mathbb{P}_{\kappa}$ the $SLE_{\kappa}$ traces as $\gamma^{\kappa}(t):=\lim_{y \to 0+}g_t^{-1}(iy+\sqrt{\kappa}W_t^{\mathbb{P}_{\kappa}})$ where $g_t$ are the maps satisfying the forward Loewner Differential Equation driven by $\sqrt{\kappa}W_t^{\mathbb{P}_{\kappa}}$. Also, for any fixed measure $\mathbb{P}_{\kappa}$ one can obtain via L\' evy's Characterization of the Brownian motion also an $SLE_{\kappa_1}$ trace $\mathbb{P}_{\kappa}$ -a.s. defined as $\gamma^{\kappa_1}_{\kappa}(t):=\lim_{y \to 0+}g_t^{-1}(iy+\sqrt{\kappa_1}W_t^{\mathbb{P}_{\kappa}})$ with $g_t$ are the maps satisfying the forward Loewner Differential Equation driven by $\sqrt{\kappa_1}W_t^{\mathbb{P}_{\kappa}}$. Once these two objects are defined, we would like to compare them under the family of measures $\mathbb{P}_{\kappa}$ for $\kappa \in \mathcal{K}\cap \mathbb{R}_+ \setminus ([0, \epsilon) \cup \{8\})$. This can be understood also as comparing the difference between the $SLE_{\kappa}$ traces generated by the canonical process under the family of measures $\mathbb{P}_{\kappa}$ and the ones generated by $\sqrt{\kappa_1}W_t^{\mathcal{P}_{\kappa}}$ with $\kappa_1$ fixed and $W_t^{\mathcal{P}_{\kappa}}$ being the Universal Brownian motion, i.e. the aggregator of the Brownian motion under the family of measures $\mathbb{P}_{\kappa}$. 
\end{remark}

\section{The quasi-sure continuity in $\kappa$ for $\kappa \in \mathcal{K}\cap \mathbb{R}_+ \setminus ([0, \epsilon) \cup \{8\})$ of the $SLE_{\kappa}$ traces}

Once the $SLE_{\kappa}$ traces  are constructed quasi-surely, we would like to prove the quasi-sure continuity in $\kappa$ of the traces.


Using  quasi-sure definition of the 
$SLE_{\kappa}$ trace allows us directly to consider uncountably many parameters $\kappa \in \mathcal{K}\cap \mathbb{R}_+ \setminus ([0, \epsilon) \cup \{8\})$. In this section, we use the notation $[\kappa_m, \kappa_M]$ for the nontrivial compact interval $\mathcal{K}\cap \mathbb{R}_+ \setminus ([0, \epsilon) \cup \{8\})$. Furthermore, we consider the following coupling: we fix a parameter $\kappa_1 \in [\kappa_m, \kappa_M]$ and we consider the canonical process on the path space under the measures $\mathbb{P}_{\kappa_1}$ and $\mathbb{P}_{\kappa}$ for $\kappa \in [\kappa_m, \kappa_M]$ and we compare the traces obtained for the fixed choice $\kappa_1$ with the family of traces obtained for $\kappa \in  [\kappa_m, \kappa_M]$, with $\kappa_m > \epsilon$, for any $\epsilon>0$.  

Let us consider $M$ to be the space of continuous curves defined on $[0,1]$ with values in the closed upper half-plane $\bar{\mathbb{H}}=\{z: \operatorname{Im} (z) \geq 0\}$. Further, we equip the space $M$ with the supremum norm. Let $l \in \mathbb{N}.$ Let us consider for any measure $\mathbb{P}_{\kappa_l}$ the corresponding $W_t^{\mathbb{P}_{\kappa_l}}$-Brownian motion, obtained from the aggregator of the Brownian motion. We further use the quasi-sure definition of the $SLE_{\kappa}$ traces obtained in the previous section. When we drive the forward Loewner differential equation with $\sqrt{\kappa_l}W_t^{\mathbb{P}_{\kappa_l}}$ we obtain $\mathbb{P}_{\kappa_l}$-a.s. the family of $SLE$ traces $\gamma^{l}(t):= \lim_{y \to 0+}g_t^{-1}(iy+\sqrt{\kappa_l}W_t^{\mathbb{P}_{\kappa_l}})$ (this quantity is defined simultaneously for all the measures $\mathbb{P}_{\kappa_l}$ in the family). Let us fix $\kappa \in [\kappa_m, \kappa_M]$ (wlog $\kappa=\kappa_m)$. Then, when we drive the Loewner equation for the fixed $\kappa_m$ with $\sqrt{\kappa_m}W_t^{\mathbb{P}_{\kappa_l}}$ we obtain $\mathbb{P}_{\kappa_l}$-a.s. the fixed $SLE_{\kappa}$ trace $\gamma_{l}^{\kappa_m}(t):= \lim_{y \to 0+}g_t^{-1}(iy+\sqrt{\kappa_m}W_t^{\mathbb{P}_{\kappa_l}})$ under all $\mathbb{P}_{\kappa_l}$ measures. We compare the curves in the sup-norm under all elements of the family of measures $\mathbb{P}_{\kappa_l}$. 

\begin{theorem}[Quasi-sure continuity in $\kappa$ of the $SLE_{\kappa}$ traces]\label{qscontinuity}
Let $\mathcal{K} \subset \mathbb{R}_+$ be a nontrivial compact interval. Then, the $SLE_{\kappa}$ traces are quasi-surely continuous in $\kappa$, for $\kappa \in [\kappa_m, \kappa_M]=\mathcal{K}\cap \mathbb{R}_+ \setminus ([0, \epsilon) \cup \{8\})$ -without loss of generality let us take $\kappa=\kappa_m$-, 
i.e. there exists a function $\theta$ with $\theta(\delta)\to 0$ as $\delta \to 0$ such that for $\kappa_M \to \kappa_m$,  for $t \in [0,1]$, we have $$||\gamma^{l}(t)-\gamma^{\kappa_m}_{l}(t)||_{\infty, [0,1]} \leq \theta( |\kappa_M -\kappa_m|),$$ quasi-surely (i.e. for all the measures $\mathbb{P}_{\kappa_l}$ with $\kappa_l \in [\kappa_m, \kappa_M]$) outside of a polar set that depends on $[\kappa_m, \kappa_M]$.

\end{theorem}

\begin{proof}[Proof of Theorem \ref{qscontinuity}]
Throughout the proof, we use the notation $F(t, y ,\kappa)=f^{(\kappa)}_t(iy)\,.$ We showed in the previous section that one can construct for the sequence of measures $\mathbb{P}_{\kappa}$ the $SLE_{\kappa}$ traces simultaneously and can view them as elements of the metric space $M$.
We use the set-up from \cite{viklund2014continuity}, in order to define the Whitney-type partition of the $(t,y,\kappa)$ space. The main idea of this section is to show how we can avoid the typical Borel-Cantelli argument of \cite{viklund2014continuity} using the quasi-sure construction of the $SLE_{\kappa}$ traces from the previous section.

We also need the following distortion result for conformal maps.

\begin{lemma}[Distortion Lemma: Lemma 2.2 in \cite{viklund2014continuity}] \label{dist}
There exists a constant $0<c<\infty$ such that the following holds. Suppose that $f_t$ satisfies the chordal Loewner PDE \ref{4} and that $z=x+iy \in \mathbb{H}$, then for $0 \leq s \leq y^2$
$$c^{-1} \leq \frac{|f_{t+s}'(z)|}{|f_{t}'(z)|} \leq c$$
and
$$ |f_{t+s}(z)-f_t(z)| \leq cy|f'_t(z)|\,.$$
\end{lemma}

We consider the partition of the $(t,y, \kappa)$ three dimensional space in boxes obtain by partitioning each coordinate. We follow the proof in \cite{viklund2014continuity} and we estimate the derivative of the map $(f_t^{(\kappa)})'(iy)$ in the corners of the boxes. Using Distortion Theorems for the conformal maps along with the following Lemma that appears in \cite{viklund2014continuity}.

\begin{lemma}[Lemma $2.3$ of \cite{viklund2014continuity}]\label{Lemvik}
Let $0<T<\infty$. Suppose that for $t \in [0,T]$, $f_t^{(1)}$ and $f_t^{(2)}$ satisfy the backward Loewner differential equation with drivers $W_t^{(1)}$ and $W_t^{(2)}\,.$ Suppose that $\epsilon=sup_{s \in [0,T]}|W_s^{(1)}-W_s^{(2)}|$.
For $u=x+iy \in \mathbb{H}$, for every $t \in [0, T]$ we have that  
$$ |f^{(1)}_t(u)-f^{(2)}_t(u)| \leq \epsilon \exp \left[\frac{1}{2}\left[ \log \frac{I_{t, y}|(f^{(1)}_{t})^{'}(u)|}{y}\log\frac{I_{t, y}|(f^{(2)}_{t})^{'}(u)|}{y}\right]^{1/2} +\log\log \frac{I_{t, y}}{y}\right]\,,$$
where $I_{t,y}=\sqrt{4t+y^2}\,.$
\end{lemma}

We use this estimate in our analysis. Namely, we take the following approach. We fix parameters $\kappa_1$ and $\kappa$ in $[\kappa_1, \kappa_M]$. Thus, in this manner here we fixed a coupling given by the choice of the initial measure $\mathbb{P}_0$ on the path space and by the relation \eqref{eq4.4.}, i.e. for each measure $\mathbb{P}_{\kappa}$ we couple to the Loewner chains with the drivers $\sqrt{\kappa_1}W_t^{\mathbb P_{\kappa}}$ and $\sqrt{\kappa}W_t^{\mathbb P_{\kappa}}$.

 Then, we can consider $W_t^{(1)}-W_t^{(2)}=\sqrt{\kappa_1}W_t^{\mathbb{P}_{\kappa}}-\sqrt{\kappa}W_t^{\mathbb{P}_{\kappa}}$. Using L\' evy's characterization of the Brownian motion, we have that under any measure $\mathbb{P}^{\kappa}$, the process $\sqrt{\kappa_1}W_t^{\mathbb{P}_{\kappa}}$ is a Brownian motion multiplied with the diffusivity constant $\sqrt{\kappa_1}$ (indeed since for any measure $\mathbb{P}_{\kappa}$, the process $\sqrt{\kappa_1}W_t^{\mathbb{P}_\kappa}$ is a local martingale with the quadratic variation $\kappa_1t$ ).
We have then that for fixed $\kappa_1$ the difference $|f^{(1)}_t(u)-f^{(2)}_t(u)|$ is a function of $\kappa$. 

Thus, we can estimate the difference using the above Lemma and the quasi-sure estimates on the derivatives of the maps $f_t(z)$, i.e. for any  $\mathbb{P}_{\kappa}$ (obtained when choosing the drivers $\sqrt{\kappa_1}W_t^{\mathbb{P}_{\kappa}}$ and $\sqrt{\kappa}W_t^{\mathbb{P}_{\kappa}}$). 
Furthermore, we have the following remark.
\begin{remark}\label{thetaremark}
Let $\beta \in \left(\frac{2}{2b_0-\kappa r_0/2},1\right)$, with $b_0=\frac{\kappa((1+4/\kappa)r_0-r_0^2)}{2} $ and $r_0=\frac{1}{4}+\frac{2}{\kappa}$ (as in the proof in the previous section). Then, it can be shown that (see Section $5.5$ in \cite{kemppainen2017schramm}) for any fixed $\kappa \neq 8$, we have that  $$\mathbb{P}(|h'_t(i2^{-j})| \geq 2^{n\beta}) \leq c2^{-(2+\epsilon_1)n},$$ where $c$ is a constant that depends on $\kappa$. 
\end{remark}

Thus, we can choose $\beta \in \left(\frac{2}{2b_0-\kappa r_0/2},1\right)$ in order to bound the derivatives of the conformal maps. In order to simplify the analysis we bound the first derivative term using Remark \ref{thetaremark} for fixed $\kappa_1$, i.e. we have $|f_t'(u)| \leq cy^{-\beta}$ $\mathbb{P}_{\kappa}$-$\text{a.s.}$, for any $\kappa \in [\kappa_m, \kappa_M]$ for $\beta \in \left(\frac{2}{2b_0-\kappa r_0/2},1\right)$.  For the other derivative term since the conformal maps are normalized at infinity there exists a constant $c<\infty$ depending only on $T$ such that $|f_t'(z)| \leq c(y^{-1}+1)$ for all $z \in \mathbb{H}$ and all $t \in [0, T]$.

 Then the estimate reads for any choice of the measure $\mathbb{P}_{\kappa}$, on complex numbers $u$ such that their imaginary parts are elements of the dyadic partition of $[0,1]$ (in order to use the estimate \eqref{finalest})

$$|f^{(1)}_t(u)-f^{(2)}_t(u) | \leq c_3 \epsilon_2(\kappa_1, \kappa) y^{-\sqrt{\frac{1+\beta}{2}}}$$
$\mathbb{P}_{\kappa}$-a.s., with $\epsilon_2(\kappa_1, \kappa)$ a function of $\kappa_1$ and $\kappa$ that tends to $0$ as $\kappa \to\kappa_1$.


\color{black}

%

Let us consider $q>0$ and $$S_{n,j,k}(q)=\left[\frac{j-1}{2^{2n}}, \frac{j}{2^{2n}} \right] \times \left[\frac{1}{2^{n}}, \frac{1}{2^{n-1}} \right] \times \left[\frac{k-1}{2^{qn}}, \frac{k}{2^{nq}} \right], $$and let
$$ p_{n,j,k}=\left(\frac{j}{2^{2n}}, \frac{1}{2^n}, \frac{k}{2^{qn}} \right) \in S_{j,n,k},$$
be the corners of the boxes. In the following, we choose $q>0$ and estimate the derivative of the the Loewner maps in corners of the boxes, as in \cite{viklund2014continuity}. Comparaed with the analysis in in \cite{viklund2014continuity}, one important aspect is that as  we change the parameter $\kappa$ (and implicitly go along the $\kappa$ axis in the Whitney boxes) we also change the measures $\mathbb{P}_{\kappa}$. In \cite{viklund2014continuity}, the typical estimate on the derivative of the map on the corners of the boxes is combined with the application of the Borel-Cantelli Lemma in order to assure the analysis on a unique nullset of the Brownian motion driving the Loewner differential equation. The use of Borel-Cantelli in this approach restricts the applicability of the derivative estimate in the corners of the boxes for the values $\kappa<2.1$ (more recently up to $\kappa < 8/3$ with new estimates in \cite{friz2019regularity}).  The novelty is that we use the polar set outside of which the aggregated solution is defined and then we can vary $\kappa \in \mathcal{K}\cap \mathbb{R}_+ \setminus ([0, \epsilon) \cup \{8\})$.  In this way, we argue that the estimates on the derivative of the maps in the corners $p_{n,j,k}$ of the Whitney boxes hold q.s. In this new setting, we avoid the restriction to the interval $\kappa\in [0, 8(2-\sqrt{3}))$, since the estimate on the derivative used in the proof of Theorem \ref{Rohdeschramm} can be used simultaneously for a family of probability measures $\mathbb{P}_{\kappa}$ for $\kappa \in \mathcal{K}\cap \mathbb{R}_+ \setminus ([0, \epsilon) \cup \{8\})$.

We give the following version of Lemma $3.3$ in \cite{viklund2014continuity}, that does not contain the restriction on the $\kappa$ interval, due to the application of the Borel-Cantelli Lemma. 

\begin{lemma}\label{corner}
Let $\epsilon>0$. Let $\kappa \in [\kappa_m, \kappa_M]=\mathcal{K}\cap \mathbb{R}_+ \setminus ([0, \epsilon) \cup \{8\})$, then q.s. there exists a random constant $c=c(\epsilon, \beta, q, \omega)<+\infty$ such that 
$|F'(p_{n,j,k})|\leq c2^{n\beta}$ for all pairs $(n, j, k) \in \mathbb{N}^3$ such that $p_{n,j,k} \in [0,1]\times[0,1]\times[\kappa_m,\kappa_M]\,.$
\end{lemma}
\begin{proof}


Using the analysis from the previous sections, we have that 
\begin{equation}\label{ser1}
\sum_{j=1}^{2^{2n}}\mathbb{P}_{\kappa}\left[|F'(p_{n,j,k})|\geq 2^{n\beta} \right] \leq c2^{-n\sigma},
\end{equation}
where the parameter $\sigma$ depends on $\kappa$. Following the analysis on the previous section, the parameter $\sigma=\sigma(\kappa)$ is such that the previous series is summable for every $ \mathbb{P}_{\kappa}$ for $\kappa \in [ \kappa_m, \kappa_M ]$. 
Following the analysis in the previous section one can show that the constant $c(\kappa)$ can be controlled in $\kappa$ in $[\kappa_m, \kappa_M]$ (see the analysis of the parameters in the previous section as well as the optimization procedure in Section $5.5$ in \cite{kemppainen2017schramm}). Thus, one can take the supremum of this constant in the interval of $\kappa$ that one considers.  
According to the analysis in the previous section, we have that the series \ref{ser1} is convergent  for all $\kappa \in \mathcal{K}\cap \mathbb{R}_+ \setminus ([0, \epsilon) \cup \{8\})$, i.e.
\begin{equation}\label{ser}
\sum_{j=1}^{2^{2n}}\sup_{\kappa}\mathbb{P}_{\kappa}\left[|F'(p_{n,j,k})|\geq 2^{n\beta} \right]
\end{equation}
is convergent for any choice of measure $\mathbb{P}_{\kappa}$ with $\kappa \in[\kappa_m, \kappa_M]=\kappa \in \mathcal{K}\cap \mathbb{R}_+ \setminus ([0, \epsilon) \cup \{8\})$.

\end{proof}

The next step is to use Distortion Theorem along with Lemma \ref{Lemvik} in order to push the estimate on the derivative from the corners of the box to all the points inside. The result is captured in the following Lemma. We emphasize that in \cite{viklund2014continuity}, there are two parts of the analysis in order to obtain this estimate for fixed $\kappa$, i.e. the analysis is split into the cases $\kappa$ near $0$ and the complementary regime. 
In our setting, we discuss only the case $\kappa> \epsilon$ for any $\epsilon>0$ since the a.s. continuity in $\kappa$ of the traces in the regime $\kappa \in [0, 8(2-\sqrt{3}))$ was proved in \cite{viklund2014continuity} already. 

\begin{lemma}
 Let $\kappa \in [\kappa_m, \kappa_M] = \mathcal{K} \cap \mathbb{R}_+ \setminus ([0, \epsilon) \cup \{8\})$, then for every $\epsilon>0$ there exists $\delta>0$ and $q>0$ and a random constant $c=c(q,\epsilon, \omega, \beta)<\infty$ such that
$diam(F(S_{n,j,k}))\leq c2^{-n\delta}\,, $ quasi-surely  for all $(n,j,k) \in \mathbb{N}^3$ with $p_{n,j,k}\in [0,1]\times [0,1]\times [\kappa_m,\kappa_M]\,.$
\end{lemma}
\begin{proof}
We will show that there exists $\delta>0$ such that 
$|F(p)-F(p_{n,j,k})| \leq cn2^{-n\delta}.$
Let us fix $\kappa_1 \in [\kappa_m, \kappa_M]$.
We estimate for $|\Delta t|\leq y^2$, using Lemmas \ref{corner} and \ref{dist}
$$|F(t+\Delta t, y, \kappa_1)-F(t, y, \kappa_1)| \leq cy|F'(p_{n,j,k})| \leq c'2^{-n(1-\beta)}$$ quasi surely with $c'=c'(\beta, q, \omega)$.

By Koebe Distortion Theorem and Lemmas \ref{corner} and \ref{dist}, we obtain that  $$|F(t+\Delta t, y+\Delta y, \kappa_1)-F(t+\Delta t, y, \kappa_1)| \leq cy|F'(p_{n,j,k})| \leq c2^{-n(1-\beta)}, $$ quasi-surely. 

Let $\phi(\beta)=\sqrt{\frac{1+\beta}{2}}$. Using Lemma \ref{Lemvik} and estimating for $\kappa=\kappa_1+\Delta \kappa$, $$\sup_{t\in [0,1]}|\sqrt{\kappa_1+\Delta\kappa }W_t^{\mathbb{P}_{\kappa}}-\sqrt{\kappa_1}W_t^{\mathbb{P}_{\kappa}}| \leq c\Delta\kappa \sup_{t\in [0,1]}|W_t^{\mathbb{P}_{\kappa}}|\leq c'\Delta \kappa\,,$$
where $c'=c'(\omega, \epsilon)<\infty,$ quasi-surely (i.e. $\mathbb{P}^{\kappa}$-a.s. for all $\kappa \in [\kappa_m, \kappa_M]$), 
we obtain that
$$|F(t+\Delta t, y+\Delta y, \kappa_1+\Delta \kappa )-F(t+\Delta t, y+\Delta y, \kappa_1)| \leq c \Delta \kappa y^{-\phi(\beta)}\log (y^{-1})\leq cn2^{-n(q-\phi(\beta))}\,,$$ quasi-surely.
We choose $\delta=\min \{1-\beta, q-\phi(\beta)\}$ that is clearly positive for the right choice $q > \phi(\beta)$, and we finish the proof.

\end{proof}

In order to finish the proof of Theorem \ref{qscontinuity}, we redo the exact elements of Theorem $4.1$ in  \cite{viklund2014continuity} under all the measures $\mathbb{P}_{\kappa}$, in our coupling. We consider the family of measures $\mathbb{P}_{\kappa}$ for $\kappa \in [\kappa_m, \kappa_M]$.
In order to achieve it we estimate under all the measures $\mathbb{P}_{\kappa}$ for $\kappa \in [\kappa_m, \kappa_M]$ using the previous lemma (i.e. the bound on the diameters of the Whitney boxes) in the following manner $|F(t, y, \kappa_1)-F(t, y , \kappa)| \leq |F(t, y, \kappa_1)-F(t, 2^{-N}, \kappa_1)|+|F(t, 2^{-N}, \kappa_1)-F(t, 2^{-N}, \kappa)|+|F(t, 2^{-N}, \kappa)-F(t, y, \kappa)| \leq C\sum_{n=N}^{\infty}2^{-n\delta}$ as we vary the parameter $\kappa \in [\kappa_m, \kappa_M]$.

When comparing the fixed value $\kappa_1$ and any other $\kappa_2 \in [\kappa_m, \kappa_M]$, (w.l.o.g $\kappa_2 > \kappa_1)$ we obtain that $|F(t, y, \kappa_1)-F(t, y , \kappa_2)| \leq |F(t, y, \kappa_1)-F(t, 2^{-N}, \kappa_1)|+|F(t, 2^{-N}, \kappa_1)-F(t, 2^{-N}, \kappa_2)|+|F(t, 2^{-N}, \kappa_2)-F(t, y, \kappa_2)| \leq C\sum_{n=N}^{\infty}2^{-n\delta} \leq C2^{-N\delta} \leq C|\kappa_1-\kappa_2|^{\delta/q},$ where we have used the stopping time $N=O(-\log|\kappa_1-\kappa_2|^{1/q})$ given by the bounds $2^{-qN} <|\kappa_1-\kappa_2| \leq 2^{-q(N-1)}.$

Then, for any choice $\kappa_2 \in [\kappa_m, \kappa_M]$, when taking $ y \to 0+$ we get 
$$|\gamma^{\kappa_1}_{\kappa_2}(t)-\gamma^{\kappa_2}(t)| \leq C|\kappa_1-\kappa_2|^{\delta/q}$$ that holds $\mathbb{P}_{\kappa_2}$ -a.s.


Let us choose without lose of generality $\kappa_m=\kappa_1$, then
one can then obtain a uniform estimate for all the family of measures $\mathbb{P}_{\kappa}$ for $\kappa \in [\kappa_m, \kappa_M]$
$$|\gamma^{\kappa_m}_{l}(t)-\gamma^{l}(t)| \leq C|\kappa_m-\kappa_M|^{\delta/q}.$$ 

This estimate holds $\mathbb{P}_{\kappa}$ a.s. for all $\kappa \in [\kappa_m, \kappa_M]$, i.e. it holds quasi surely for a family of probability measures indexed by $\kappa \in [\kappa_m, \kappa_M]$. 
Thus, outside a polar set that depends on the choice of the nontrivial compact interval $[\kappa_m, \kappa_M]$, we obtain the desired result.
Taking, $\kappa_M \to \kappa_m$, we obtain the desired result.

\end{proof}


\newpage 

\bibliographystyle{plain}
\bibliography{literature.bib}

\begin{thebibliography}{10}

\bibitem{artzner1999coherent}
Philippe Artzner, Freddy Delbaen, Jean-Marc Eber, and David Heath.
\newblock {Coherent measures of risk}.
\newblock {\em Mathematical finance}, 9(3):203--228, 1999.

\bibitem{astala2011random}
Kari Astala, Antti Kupiainen, Eero Saksman, and Peter Jones.
\newblock {Random conformal weldings}.
\newblock {\em Acta mathematica}, 207(2):203--254, 2011.

\bibitem{berestycki2014lectures}
Nathaniel Berestycki and James Norris.
\newblock {Lectures on Schramm--Loewner Evolution}.
\newblock {\em Lecture notes, available on the webpages of the authors}, 2014.

\bibitem{bieberbach1916uber}
Ludwig Bieberbach.
\newblock {Uber die Koeffizienten derjenigen Potenzreihen, welche eine
  schlichte Abbildung des Einheitskreises vermitteln}.
\newblock {\em Sitzungsberichte Preussische Akademie der Wissenschaften},
  138:940--955, 1916.

\bibitem{bouleau2010dirichlet}
Nicolas Bouleau and Francis Hirsch.
\newblock {\em {Dirichlet forms and analysis on Wiener space}}, volume~14.
\newblock Walter de Gruyter, 2010.

\bibitem{de1985proof}
Louis De~Branges.
\newblock {A proof of the Bieberbach conjecture}.
\newblock {\em Acta Mathematica}, 154(1-2):137--152, 1985.

\bibitem{denis2006theoretical}
Laurent Denis and Claude Martini.
\newblock {A theoretical framework for the pricing of contingent claims in the
  presence of model uncertainty}.
\newblock {\em The Annals of Applied Probability}, 16(2):827--852, 2006.

\bibitem{follmer2002convex}
Hans F{\"o}llmer and Alexander Schied.
\newblock {Convex measures of risk and trading constraints}.
\newblock {\em Finance and stochastics}, 6(4):429--447, 2002.

\bibitem{friz2017existence}
Peter~K Friz and Atul Shekhar.
\newblock {On the existence of SLE trace: finite energy drivers and
  non-constant $\kappa$}.
\newblock {\em Probability Theory and Related Fields}, 169(1-2):353--376, 2017.

\bibitem{friz2019regularity}
Peter~K Friz, Huy Tran, and Yizheng Yuan.
\newblock {Regularity of the Schramm-Loewner field and refined
  Garsia-Rodemich-Rumsey estimates}.
\newblock {\em arXiv preprint arXiv:1906.11726}, 2019.

\bibitem{karandikar1995pathwise}
Rajeeva~L Karandikar.
\newblock {On pathwise stochastic integration}.
\newblock 1995.

\bibitem{karrila2018limits}
Alex Karrila.
\newblock {Limits of conformal images and conformal images of limits for planar
  random curves}.
\newblock {\em arXiv preprint arXiv:1810.05608}, 2018.

\bibitem{kemppainen2017schramm}
Antti Kemppainen.
\newblock {Schramm--Loewner Evolution}.
\newblock In {\em Schramm--Loewner Evolution}, pages 69--100. Springer, 2017.

\bibitem{kemppainen2012random}
Antti Kemppainen and Stanislav Smirnov.
\newblock {Random curves, scaling limits and Loewner evolutions}.
\newblock {\em arXiv preprint arXiv:1212.6215}, 2012.

\bibitem{lawler2008conformally}
Gregory~F Lawler.
\newblock {\em {Conformally invariant processes in the plane}}.
\newblock Number 114. American Mathematical Soc., 2008.

\bibitem{lawler2011conformal}
Gregory~F Lawler, Oded Schramm, and Wendelin Werner.
\newblock {Conformal invariance of planar loop-erased random walks and uniform
  spanning trees}.
\newblock In {\em Selected Works of Oded Schramm}, pages 931--987. Springer,
  2011.

\bibitem{lind2010collisions}
Joan Lind, Donald~E Marshall, and Steffen Rohde.
\newblock {Collisions and spirals of Loewner traces}.
\newblock {\em Duke Mathematical Journal}, 154(3):527--573, 2010.

\bibitem{lowner1923untersuchungen}
Karl L{\"o}wner.
\newblock {Untersuchungen \"uber schlichte konforme Abbildungen des
  Einheitskreises. I}.
\newblock {\em Mathematische Annalen}, 89(1-2):103--121, 1923.

\bibitem{revuz2013continuous}
Daniel Revuz and Marc Yor.
\newblock {\em {Continuous martingales and Brownian motion}}, volume 293.
\newblock Springer Science \& Business Media, 2013.

\bibitem{schramm2000scaling}
Oded Schramm.
\newblock {Scaling limits of loop-erased random walks and uniform spanning
  trees}.
\newblock {\em Israel Journal of Mathematics}, 118(1):221--288, 2000.

\bibitem{schramm2005basic}
Oded Schramm and Steffen Rohde.
\newblock {Basic properties of SLE}.
\newblock {\em Annals of mathematics}, 161(2):883--924, 2005.

\bibitem{sheffield2012strong}
Scott Sheffield and Nike Sun.
\newblock {Strong path convergence from Loewner driving function convergence}.
\newblock {\em The Annals of Probability}, pages 578--610, 2012.

\bibitem{shekhar2017remarks}
Atul Shekhar and Yilin Wang.
\newblock {Remarks on Loewner Chains Driven by Finite Variation Functions}.
\newblock {\em arXiv preprint arXiv:1710.07302}, 2017.

\bibitem{soner2011quasi}
Mete Soner, Nizar Touzi, and Jianfeng Zhang.
\newblock {Quasi-sure stochastic analysis through aggregation}.
\newblock {\em Electronic Journal of Probability}, 16:1844--1879, 2011.

\bibitem{songanother}
L~Song and P~Y Wu.
\newblock {Another Borel-Cantelli lemma for capacities [J]}.
\newblock {\em Chinese J. Appl}.

\bibitem{viklund2014continuity}
Fredrik~Johansson Viklund, Steffen Rohde, and Carto Wong.
\newblock {On the continuity of SLE $\kappa$ in $\kappa$}.
\newblock {\em Probability Theory and Related Fields}, 159(3-4):413--433, 2014.

\end{thebibliography}

\end{document}